\documentclass{amsart}

\usepackage{amsfonts}
\usepackage{amsmath}
\usepackage{amsthm}
\usepackage{amssymb}
\usepackage{latexsym}
\usepackage{enumerate}

\newtheorem{theorem}{Theorem}[section]

\newtheorem{proposition}[theorem]{Proposition}
\newtheorem{corollary}[theorem]{Corollary}

\newtheorem{remark}[theorem]{Remark}

\theoremstyle{definition}
\newtheorem{definition}[theorem]{Definition}

\numberwithin{equation}{section}

\begin{document}

\newcommand{\bb}{\mathfrak{b}}
\newcommand{\cc}{\mathfrak{c}}
\newcommand{\N}{\mathbb{N}}
\newcommand{\R}{\mathbb{R}}
\newcommand{\A}{{\mathbb{R}}_+\cup\{0\}}
\newcommand{\forces}{\Vdash}
\newcommand{\LL}{\mathbb{L}}
\newcommand{\K}{\mathbb{K}}
\newcommand{\Q}{\mathbb{Q}}

\newcommand{\To}{\longrightarrow}

\title{A continuous image of a Radon-Nikod\'{y}m compact space which is not Radon-Nikod\'{y}m}
\author{Antonio Avil\'{e}s}
\thanks{First author by was 
supported by MEC and FEDER (Project MTM2008-05396), Fundaci\'{o}n S\'{e}neca (Project 08848/PI/08), {\em Ramon y Cajal} contract (RYC-2008-02051) and an FP7-PEOPLE-ERG-2008 action.}
\email{avileslo@um.es}
\address{Departamento de Matem\'{a}ticas, Universidad de Murcia, 30100 Murcia (Spain)}

\author{Piotr Koszmider}
\thanks{The second author was partially supported by the National Science Center research grant DEC-2011/01/B/ST1/00657.  He also expresses his gratitude to the Functional Analysis group in Murcia
for constant and ongoing support which included organizing several visits to the University of Murcia and made this research possible.} 
\email{P.Koszmider@Impan.pl}
\address{Institute of Mathematics, Polish Academy of Sciences,
ul. \'Sniadeckich 8,  00-956 Warszawa, Poland}

%
\subjclass{}
%
%
%
\begin{abstract} 
We construct a continuous image of a Radon-Nikod\'{y}m compact space which is not Radon-Nikod\'{y}m compact, solving the problem posed in the 80ties by Isaac Namioka. 
\end{abstract}

\maketitle

\markright{}

\section{Introduction}

Recall that a Banach  space $X$ has the Radon-Nikod\'{y}m property if and only if the Radon-Nikod\'{y}m theorem
holds for vector measures with values in $X$ (see \cite{DiestelUhl}). This property plays a central role in
the theory of vector measures. 
It has been clear for long time that dual Banach spaces with the Radon-Nikod\'{y}m property and their
weak$^*$ compact subsets play special role in this theory  \cite{Reynov,StegallRNII}.
Isaac Namioka \cite{Namioka} defined a compact space to be  Radon-Nikod\'{y}m compact (or RN for short)
if and only if it is homeomorphic to a weak$^\ast$ compact subset of a dual  Banach space with the Radon-Nikod\'{y}m property. For example, as reflexive Banach spaces are dual spaces with  the Radon-Nikod\'ym property,
the results of \cite{factor} imply that Eberlein compact spaces are RN compact.
Already in \cite{Namioka} a number of interesting properties of RN compacta are proven,
as well as an elegant internal characterization of RN compacta is given. 
The investigation of this class of compact spaces continued later, with some remarkable results like the relation with Corson and Eberlein compact spaces \cite{OriSchVal, StegallRN}.

But the question which has attracted more attention and produced a larger literature on
RN compacta is the following very basic problem, already posed in \cite{Namioka}
and traced in \cite{FabHeiMat} to \cite{Gr}  which has remained open up to this date:\\

\begin{center}

\emph{Is the class of RN compact spaces closed under continuous images?}

\ \\
\end{center}

A number of partial positive results to the above question of continuous images of RN compacta have been proven. If $L$ is a continuous image of an RN compact space, then $L$ is RN compact if any of the following conditions hold: 
\begin{enumerate}
\item $L$ is almost totally disconnected  \cite{Arvanitakis}, meaning that $L\subset [0,1]^I$ and for every $x\in L$, $|\{i\in I : x_i\in (0,1)\}|\leq \omega$. This includes in particular the cases when $L$ is zero-dimensional (attributed independently to Reznichenko \cite{ArkhangelskiiGT2}) and when $L$ is Corson \cite{StegallRN}, and less obviously also the case when $L$ is linearly ordered \cite{RNorder}.
\item The weight of $L$ is less than cardinal $\mathfrak{b}$ \cite{RNcompact}.
\item $L$ is the union of two RN compact subspaces $L = L_1\cup L_2$ and some special hypothesis hold, like $L_1\cap L_2$ being metrizable, $G_\delta$ or scattered, or when $L_1$ is a retract of $L$ or when $L\setminus L_1$ is scattered \cite{MatSte}.
\end{enumerate} 

Other articles devoted to the problem of the continuous image include \cite{ArvAvi, FabHeiMat, IanWat, Namiokanote}. More information can be found in \cite{AviKal, FabianWA, Fabiansurvey,Namiokasurvey}, which are dedicated to the topic, or contain sections dedicated to it. The purpose of this article is to provide a negative solution to the general problem:

\begin{theorem}\label{RNtheorem}
There exists a continuous surjection $\pi:\LL_0\To \LL_1$ such that $\LL_0$ is a zero-dimensional RN compact space but $\LL_1$ is not RN compact.
\end{theorem}

This contrasts with other similar classes of compact spaces arising in functional analysis, like Eberlein compacta (weakly compact subsets of Banach spaces) or Corson compacta (compact subsets of $\Sigma$-products), for which the stability under continuous images happened to be a nontrivial fact, but was finally shown to hold true in \cite{BRW} and \cite{CorsonImages} respectively.
The class of RN compact spaces, on the other hand, does show other permanence properties present also for Eberlein  compact
spaces and many other classes of compact spaces playing important roles in Banach space theory. Namely, there
is an isomorphism invariant class of Banach spaces (of Asplund generated spaces) associated
with it  in the sense that if
 $K$ is an RN compact  then, the space $C(K)$ of real valued continuous functions on $K$ is an Asplund generated space, and if $X$ is
an Asplund generated space, then the dual ball $B_{X^*}$ is RN compact.

A version of the above question on continuous images of RN compacta
 on the Banach space level, i.e., if subspaces of Asplund generated
spaces are Asplund generated was answered in the negative already 30 years ago in \cite{StegallRNII}.
In this language our result is equivalent to constructing a  subspace $Y\subseteq X$
of an Asplund generated space $X$ such that the dual ball $B_{Y^*}$ is not RN compact (see 
\cite{FabHeiMat}). Note that Stegall's argument from \cite{StegallRNII} is far from achieving this,
as it uses Rosenthal's non WCG subspace of a WCG space from \cite{rosenthal}, but
by \cite{BRW}
the dual unit ball of the subspace is even an Eberlein compactum and so RN compact. The point here is
that $B_{X^*}$ may be RN compact for $X$ not Asplund generated but
$B_{{C(K)^*}}$ is RN compact if and only if $K$ is RN compact if and only if $C(K)$ is Asplund generated.
We also have a similar chain of equivalences for RN replaced by a continuous image of
RN and Asplund generated replaced by  a subspace of Asplund generated (see 
\cite{FabianWA,FabHeiMat, sigmaasplund}).
It follows that both classes of RN compact spaces and their continuous images are stable under taking isomorphism of their space of continuous functions, meaning that
\begin{enumerate}
\item If $L$ is RN compact and $C(K)$ is isomorphic to $C(L)$, then $K$ is also RN compact.
\item  If $L$ is a continuous image of an RN compact space and $C(K)$ is isomorphic to $C(L)$, then $K$ is also a continuous image of an RN compact space.
\end{enumerate}

 Now, if we combine these facts with the already mentioned result that an almost
 totally disconnected image of an RN compactum is RN compact, we obtain a remarkable consequence of our example:

\begin{corollary}
The space $C(\LL_1)$ is not isomorphic to any $C(K)$ where $K$ is almost totally disconnected.
\end{corollary}

The question whether there could exist a compact space $L$ such that $C(L)$ is not isomorphic to any $C(K)$ with $K$ totally disconnected has been a long standing open problem motivated by the
Bessaga Milutin Pe\l czy\'nski  classification of 
separable Banach spaces of the form $C(K)$. It was first solved in the negative by the second author in \cite{Koszmider}. However the example obtained there (and others which have been constructed later with similar techniques like in \cite{Plebanek}) is very different from this one, because in that case $C(L)$ was
an indecomposable Banach space. This in particular means, on the level of 
compact space  $L$, that it contains no convergent sequences
and is strongly rigid (all nonidentity continuous maps from $L$ into itself are constant) as shown in
\cite{irina}. Moreover the dual ball $B_{C(L)^*}$ with the weak$^*$ topology satisfies a strong
rigidity condition (23 of \cite{Koszmiderracsam}).
However, our space $\LL_1$ has many nontrivial continuous transformations into itself and
as  a continuous image of an RN compactum, it is sequentially compact \cite{Namioka} and hence $C(\LL_1)$ contains many infinite-dimensional co-infinite-dimensional complemented subspaces. Thus the fact that a $C(K)$ space is not isomorphic to any $C(L)$ for $L$
totally disconnected does not imply properties of spaces from \cite{Koszmider} like indecomposability or
not being isomorphic to its hyperplanes and the geometry of such a space can be quite nice.\\

Let us now explain the main idea of our construction. For this we need a bunch of definitions.

\begin{definition}\label{maindefinitions} Let $K$ be a topological space and $d:K^2\rightarrow {{\mathbb{R}}_+\cup\{0\}}$
be a metric on the set $K$ (not related to the topology on $K$).
\begin{enumerate}
\item We say that $d$ fragments $K$ if and only if for every $\varepsilon>0$ and every closed $F\subseteq K$
there is an open $U\subseteq K$ such that $U\cap F\neq\emptyset$ and 
$$diam_d(U\cap F)=\sup\{d(x,y) : x,y\in U\cap F\}<\varepsilon.$$
\item If $K',K''\subseteq K$, then $d(K', K'')=\inf\{d(x,y): x\in K', y\in K''\}$,
\item We say that $d$ is lower semi continuous (l.s.c.) if and only if given distinct $x, y\in K$
and $0<\delta<d(x,y)$,  there are open $U\ni x$ and $V\ni y$ such that $d(U,V)>\delta$.
\item We say that $d$ is Reznichenko  if and only if given distinct $x, y\in K$
  there are open $U\ni x$ and $V\ni y$ such that $d(U,V)>0$.
\end{enumerate} 
\end{definition}

Fragmentability was formally introduced in \cite{JayRog} and its relation to RN compacta comes from the fact that every bounded subset of a dual space with the Radon-Nikod\'{y}m property is fragmented by the dual norm \cite{NamPhe,StegallIsrael}.
A compact space $K$ is an RN compact space if and only if there is  an l.s.c. metric on $K$ which fragments $K$ \cite{Namioka}. Compact spaces which are fragmented by a Reznichenko metric constitute a superclass of RN compact spaces, sometimes called strongly fragmentable compact spaces \cite{Fabiansurvey,Namiokasurvey}, but which coincides with the class of quasi RN compact spaces introduced by  Arvanitakis \cite{Arvanitakis} by a result of Namioka \cite{Namiokanote} (cf. also \cite{Fabiansurvey}). What we need to know about quasi RN compacta is
that the above mentioned result of Arvanitakis  applies to them, that is, totally disconnected
quasi RN compacta are RN compacta \cite{Arvanitakis}.

The main insight that leads to  the construction is to see how to destroy the l.s.c. property of a metric
without destroying the Reznichenko property. This is described in Propositions \ref{preserving},
\ref{destroying}. It is done by a ``smart" replacement of some point by the  unit interval
and can be interpreted as an operation of the so called resolution of a topological space. A central role
of this method in topology is claimed in \cite{Watson} where it is traced back to \cite{fedorchuk}.
It is probably not a coincidence that the spaces constructed in \cite{Koszmider} can also be viewed as
obtained by versions of resolutions. We start with an RN compactum which is
simple modifications of appropriate scattered space of height 3, just to make our resolutions
powerful enough.
Then we carefully do as many resolutions of nonisolated points as necessary to destroy all l.s.c. metrics i.e.,
to make sure that the resulted space is not  RN compact. We need to predict all these l.s.c. metrics
using a combinatorial or a descriptive set-theoretic tool. Finally
it turns out that not only the space
remains with a Reznichenko metric after all these resolutions but also 
its standard totally disconnected preimage maintains through the resolutions
 a metric which fragments it. So, it is enough to use the above mentioned result of Arvanitakis to conclude that
this totally disconnected preimage is  RN compact. 

The structure of the paper is as follows: In Section \ref{notation} we introduce some basic notation. In Section \ref{basicspace} we present what we call a \emph{basic space}, the starting point of our construction. In Section \ref{proofsection} we explain how to obtain a surjection $\pi:\LL_0\To \LL_1$ like in Theorem \ref{RNtheorem} from a basic space. Finally in Sections \ref{ZFCexample} and \ref{diamondexample} we provide two different ways of constructing a basic space. The first one is based on a version of the Ciesielski-Pol compact space \cite{CiePol} and can be done within ZFC without additional axioms. The second construction is based on ladder systems on $\omega_1$  (see \cite{shura}, \cite{pollindelof}) and assumes $\diamondsuit$. We found of interest to include the  construction under $\diamondsuit$ as well because it has additional properties, for instance separable subspaces of $\LL_0$ and $\LL_1$ are metrizable.\\

The compact spaces that we construct have weight $\mathfrak c$ but we do not know if perhaps $\mathfrak{b}$ is the optimal weight of a counterexample to the problem. The reader can find in \cite{AviKal} and \cite{Fabiansurvey} a number of interesting problems on RN compacta that still remain open. For example, we may mention that it is unknown if every RN compact space is the continuous image of a zero-dimensional RN compact space, or if it is always homeomorphic to a subspace of the space of probability measures on a scattered space. We do not know as well whether the class of continuous images of RN compact spaces coincides with that of  quasi RN compact spaces. It would be also interesting to find counterexamples to restricted forms of the continuous image problem, like the union of two RN compact spaces not to be RN, or the convex hull of an RN compact space not to be RN.

\section{Some notations}\label{notation}

By $\Delta = 2^\mathbb{N}$ we will denote the Cantor set, the set of all infinite sequences of 0's and 1's endowed with the  topology induced by the metric $\rho:\Delta\times\Delta\To \mathbb{R}$ given by
$$\rho(x,y) = 2^{-\max\{k: x_k\neq y_k\}}$$

By $T=2^{<\omega}$ we denote the set of all finite sequences of 0's and 1's. 
For $t\in T$ by $|t|$ we denote the cardinality of $t$, that is, its length.
If $t= (t_1,\ldots,t_n)\in T$ and $s=(s_1,\ldots)\in T\cup \Delta$, we denote $t^\frown s = (t_1,\ldots,t_n,s_1,s_2,\ldots)$. If $t=(t_1,t_2\ldots)\in T\cup \Delta$ and $s=(s_1,s_2,\ldots)\in T\cup \Delta$, $t<s$ refers to the lexicographical order, so it means that there exists $k$ such that $t_k<s_k$ but $t_i=s_i$ for $i<k$.\\

Given $s,t\in T$, we consider the continuous function $\Gamma_s^t:\Delta\To \Delta$ defined as:
\begin{itemize}
\item $\Gamma_s^t(z) = t^\frown (0,0,0,\ldots)$ if $z < s$,

\item $\Gamma_s^t(s^\frown \lambda) = t^\frown \lambda$ for every $\lambda\in \Delta$,

\item $\Gamma_s^t(z) = t^\frown (1,1,1,\ldots)$ if $z > s$.\\
\end{itemize}

The function $q:\Delta\To [0,1]$ is the standard continuous surjection given by $$q(t_1,t_2,\ldots) = \sum_{k=1}^{\infty}\frac{t_k}{2^k}$$
Notice that $q$ transfers the lexicographical order of $\Delta$ to the usual order of $[0,1]$, in the sense that $x\leq y$ implies that $q(x)\leq q(y)$.

\section{The starting basic space}\label{basicspace}

We shall call a \emph{basic space} a compact scattered space $K$ which can be written as $K = \bigcup_{n\in\mathbb{N}}A_n \cup B \cup C$ satisfying the following properties

\begin{enumerate}
\item All points of $A=\bigcup_n A_n$ are isolated in $K$.
\item For every $x\in B$ there exists an infinite set $C_x\subset A$ such that $\overline{C_x} = C_x\cup\{x\}$ and moreover, $\overline{C_x}$ is open in $K$.
\item There exists a function $\psi:B\To \mathbb{N}^\mathbb{N}$ such that: Given any family $\{X_m^n : m,n\in\mathbb{N}\}$ of subsets of $A$ with $A_n = \bigcup_m X_m^n$ for every $n$, there exists $x\in B$ such that\footnote{we denote by $\psi(x)[n]$ the evaluation on $n$ of the function $\psi(x):\mathbb{N}\To\mathbb{N}$}  $C_x\cap  X^n_{\psi(x)[n]}$ is infinite for all $n$.
\end{enumerate}

\section{How to obtain the desired continuous image from a basic space}\label{proofsection}

The first step is to consider the compact space $L$ obtained from the basic space $K$ by substituting each point of $A$ by a copy of the Cantor set $\Delta$. That is,

$$ L = (A\times \Delta) \cup B \cup C$$

A basic neighborhood of a point $(a,t)$ is of the form $\{a\}\times U$ where $U$ is a neighborhood of $t$ in $\Delta$. A basic neighborhood of a point $x\in B\cup C$ is of the form $((U\cap A)\times \Delta) \cup U\setminus A$, where $U$ is a neighborhood of $x$ in $K$.\\

We shall use the countable set $T=2^{<\omega}$ instead of $\mathbb{N}$ in order to describe the basic space $K$. So we shall write $A=\bigcup_{t\in T}A_t$ instead $A = \bigcup_{n\in\mathbb{N}}A_n$, and the last condition on our basic space will be now read as:\\

\begin{itemize}
\item[(3')] There exists a function $\psi:B\To T^T$ such that: Given any family $\{X_s^t : s,t\in T\}$ of subsets of $A$ with $A_t = \bigcup_s X_s^t$ for every $t$, there exists $x\in B$ such that $C_x\cap X^t_{\psi(x)[t]}$ is infinite for all $t\in T$.\\
\end{itemize}

For every $x\in B$ we consider a continuous function $g_x: L\setminus \{x\} \To \Delta$ defined in the following way:
\begin{enumerate}
\item $g_x(y) = 0$ whenever $y\not\in C_x\times \Delta$, $y\neq x$,
\item $g_x(a,z) = \Gamma_{\psi(x)[t]}^t(z)$ for $a\in A_t\cap C_x$, $z\in\Delta$.
\end{enumerate} 
We also consider $f_x:L\setminus\{x\}\To [0,1]$, $f_x = q\circ g_x$.\\

Now, we are in a position to define the announced $\pi:\LL_0\To \LL_1$. Let

$$\LL_0 = \left\{[u,v] \in L\times \Delta^B : g_x(u) = v_x \text{ for all } x\in B\setminus\{u\}\right\}$$
$$\LL_1 = \left\{[u,v] \in L\times [0,1]^B : f_x(u) = v_x \text{ for all } x\in B\setminus\{u\}\right\}$$
$$\pi[u,v] = [u, q(v_x)_{x\in B}].$$

Notice an important fact about the structure of $\LL_0$ and $\LL_1$. When $u\in L\setminus B$, there is a unique point of $\LL_i$ of the form $[u,v]$. However, when $u\in B$, the set $\{[u,v]\in \LL_i\}$ is homeomorphic to $\Delta$ when $i=0$ and to $[0,1]$ when $i=1$, because all coordinates $v_x$ are determined by $u$ as $v_x = g_x(u)$ (or $f_x(u)$) when $x\neq u$, but $v_u$ can take any value from $\Delta$ (or $[0,1]$). In this way, we can think that we have \textit{splitted} each point of $B$ into a Cantor set (or into an interval) \textit{following} the functions $g_x$ (or respectively $f_x$).

\begin{proposition}\label{preserving}
$\LL_0$ is  RN compact.
\end{proposition}

\begin{proof}
We check first that $\LL_0$ is closed in $K\times \Delta^B$, hence compact. So fix $[u,v]\in K\times \Delta^B \setminus \LL_0$ and we find a neighborhood of $[u,v]$ disjoint from $\LL_0$. Since $[u,v]\not\in \LL_0$, there exists $x\neq u$ such that $g_x(u) \neq v_x$. Let $V$ and $W$ be disjoint open neighborhoods in $\Delta$ of $g_x(u)$ and $v_x$ respectively. Let $U$ be a neighborhood of $u$ in $L$ such that $g_x(U)\subset V$ and $x\not\in U$. The neighborhood we are looking for is $$\tilde{\mathbb{U}} = \{[u',v']\in L\times \Delta^B : u'\in U, v'_x\in W\}.$$ 
Indeed, if $[u',v']\in\tilde{\mathbb{U}}$, then $x\neq u'$ since $x\not\in U$, but $g_x(u')\neq v'_x$ because $g_x(u')\in V$ while $v'_x\in W$.\\ 

Consider the following metric $d:\LL_0\times \LL_0 \To [0,1]$:
\begin{enumerate}
\item $d([u,v],[u,v]) = 0$,
\item $d([u,v],[u,v']) = \rho(v_u,v'_u)$ if $u\in B$,
\item $d([u,v],[u',v']) = \rho(r,r')$ if $u,u'\in A\times\Delta$, $u=(a,r)$, $u'=(a,r')$,
\item $d([u,v],[u',v']) = 1$ in any remaining case when $[u,v]\neq [u',v']$.\\
\end{enumerate}

\noindent\emph{Claim 1.} The metric $d$ fragments $\LL_0$.\par
\noindent Recall the definition of fragmentability (1) \ref{maindefinitions} and consider a  nonempty $Y\subset \LL_0$,

\begin{enumerate}
\item If $Y$ contains a point of the form $[u,v]$, with $u=(a,r)\in A\times \Delta$, then take $U$ a neighborhood of $r$ in $\Delta$ of $\rho$-diameter less than $\varepsilon$, and then $V=\{[u,v] : u=(a,s), s\in U\}$ is a neighborhood of $[u,v]$ of $d$-diameter less than $\varepsilon$.

\item If $Y$ does not contain any point as in the previous case, then $u\in B\cup C$ for all $[u,v]\in Y$. Since $B\cup C\subset K$ is scattered, we can find $u^0$ an isolated point of the set $Z = \{u\in B\cup C : \exists v\ [u,v]\in Y\}$. Suppose $[u^0,v^0]\in Y$, let $U$ be a neighborhood of $u^0$ in $L$ that isolates $u^0$ inside $Z$, and $W$ a neighborhood of $v^0_{u^0}$ in $\Delta$ of $\rho$-diameter less than $\varepsilon$. Then
$$V = \{[u,v]\in Y : u\in U, v_{u^0}\in W\} \subset \{[u_0,v] : v_{u^0}\in W\}$$
is a nonempty relative open subset of $Y$ of $d$-diameter less than $\varepsilon$.\\
\end{enumerate}

\noindent\emph{Claim 2.} The metric $d$ is a Reznichenko metric. \par
\noindent By (4) of \ref{maindefinitions} to prove that $d$ is Reznichenko, given $[u^0,v^0]\neq [u^1,v^1]$, we must find neighborhoods $U$ and $V$ of $[u^0,v^0]$ and $[u^1,v^1]$ respectively such that
$$d(U,V) = \inf\{d(z,z') : z\in U, z'\in V\}>0$$
We distinguish several cases:
\begin{enumerate}
\item If $u^0,u^1\in A\times \Delta$, $u^0=(a,r)$, $u^1=(a,r')$, then we can take $J$ and $J'$ neighborhoods of $r$ and $r'$ respectively at positive $\rho$-distance, and then take $U = \{[(a,s),v]\in\LL_0 : s\in J\}$ and $V = \{[(a,s),v]\in\LL_0 : s\in J'\}$.
\item In any other case when $u^0\neq u^1$, we can take neighborhoods $G$ and $G'$ of $u^0$ and $u^1$ such that $d([u,v],[u',v']) = 1$ whenever $u\in G$ and $u'\in G'$.
\item If $u^0 = u^1= x \in B$, we consider $G$ and $G'$ disjoint clopen neighborhoods of $v^0_{x}$ and $v^1_{x}$ respectively inside $\Delta$. Let $W = (C_x\times \Delta)\cup \{x\}$ which is a clopen neighborhood of $x$ in $L$. We claim that
$U = \{[u,v] : u\in W, v_x\in G\}$ and $V =\{[u,v] : u\in W, v_x\in G'\}$ are at a positive $d$-distance as required. If they were not, we could find sequences $e_n\in U$ and $\tilde{e}_n\in V$ such that $d(e_n,\tilde{e}_n)\rightarrow 0$. We can suppose that $e_n=[(a_n,z_n),v^n]$, $\tilde{e}_n = [(a_n,\tilde{z}_n),\tilde{v}^n)]$, the $a_n$ is the same in both cases since otherwise $d(e_n,\tilde{e}_n)=1$. By passing to a subsequence we can suppose that $v^n_x\rightarrow w\in G$ and $\tilde{v}_x^n\rightarrow \tilde{w}\in G'$. By passing to a further subsequence, we can reduce this case to  one of the following two subcases:
\begin{enumerate}
\item either there is a $t$ such that $a_n\in A_t$ for all $n$. Then $v^n_x = \Gamma^t_{\psi(x)[t]}(z_n)$ and $\tilde{v}^n_x = \Gamma^t_{\psi(x)[t]}(\tilde{z}_n)$. Since $\rho(z_n,\tilde{z}_n) = d(e_n,\tilde{e}_n)\rightarrow 0$, the continuity of $\Gamma^t_{\psi(x)[t]}$ implies that $\rho(w,\tilde{w})=0$, a contradiction. 
\item or $a_n\in A_{t_n}$ and $|t_n|\rightarrow \infty$. In that case, the $\rho$-diameter of $\Gamma^{t_n}_{\psi(x)[t_n]}(\Delta) = \{t_n^\frown \lambda : \lambda\in \Delta\}$ tends to 0 as well. Again, this implies that $\rho(v^n_x,\tilde{v}^n_x)\rightarrow 0$ and $w=\tilde{w}$, a contradiction.

\end{enumerate}

\end{enumerate}
\ \\

Every compact space fragmented by a Reznichenko metric is quasi-RN \cite{Namiokanote}, and every zero-dimensional quasi-RN compact
space is RN compact \cite{Arvanitakis}.

\end{proof}

\begin{remark}{\rm There is  a quite natural way of redefining the metric $d$ on the pairs
$[u,v], [u, v']$ for $u\in B$ and $v\in \Delta$ to obtain an l.s.c. quasimetric (see \cite{Arvanitakis}) on $\LL_0$ which could give another 
proof of the RN property following the results of \cite{Arvanitakis}}
\end{remark}

\begin{proposition}\label{destroying}
$\LL_1$ is not  RN compact.
\end{proposition}

\begin{proof}
First, $\LL_1$ is compact being a continuous image of $\LL_0$. If $\LL_1$ is RN compact, then there exists a lower semicontinuous metric $\delta:\LL_1\times\LL_1\To\mathbb{R}$ which fragments $\LL_1$. Given $a\in A$ and $z\in\Delta$ let us denote by $a+z$ the unique point of $\LL_1$ of the form $a+z = [(a,z),v]$. By the fragmentability condition, whenever $a\in A_t$ we can find $ s(a)\in T$ such that 
$$\delta\left( a+s(a)^\frown(0,0,\ldots), a + s(a)^\frown(1,1,\ldots)\right)< \frac{1}{4^{|t|}}.$$
Let $X_s^t = \{a\in A_t : s(a)=s\}$, so that $A_t = \bigcup_{s\in T}X_s^t$ for every $t\in T$. We are in the position to apply the fundamental property (3') of our basic space, so that we can find $x\in B$ such that
$C_x \cap X^t_{\psi(x)[t]}$ is infinite for all $t\in T$. This means that for every $t\in T$ we can find an infinite sequence $\{a_n\}\subset C_x\cap A_t$ such that $s(a_n) = \psi(x)[t]$ for every $n$. Now, for every $\xi\in [0,1]$ let us denote $x\oplus \xi = [x,v]\in\LL_1$, where $v_x = \xi$ and $v_y = f_y(x)$ for $y\in B\setminus \{x\}$. If we remember the definition of $f_x$ and $g_x$, we notice that
$$f_x\left( a_n + \psi(x)[t]^\frown (0,0,0,\ldots) \right) = q(t^\frown(0,0,0,\ldots)) =: t^0$$
$$f_x\left( a_n + \psi(x)[t]^\frown (1,1,1,\ldots) \right) = q(t^\frown(1,1,1,\ldots)) =: t^1$$

Now, taking limits when $n\rightarrow \infty$, $a_n + \psi(x)[t]^\frown (i,i,i,\ldots)\rightarrow x\oplus\xi^i$ where, by looking at the $x$-coordinate, $$\xi^i= \lim_n  f_x\left( a_n + \psi(x)[t]^\frown (i,i,i,\ldots) \right) = q(t^\frown(i,i,i,\ldots)) = t^i$$
Using the lower semicontinuity\footnote{We are using the following property of a lower semicontinuous metric, which is a direct consequence of Definition \ref{maindefinitions}: if $x_n\to x$, $y_n\to y$ and $\delta(x_n,y_n)\leq \varepsilon$ for every $n$, then $\delta(x,y)\leq \varepsilon$.} of $\delta$, we conclude that
$$\delta(x\oplus t^0,x\oplus t^1) \leq \frac{1}{4^{|t|}}$$
and this happens for every $t\in T$. Now fix $m\in\mathbb{N}$, and observe that
$$\left\{ [t^0,t^1] : t\in T,\ |t| = m\right\} = \left\{ [(k-1)2^{-m}, k2^{-m}] : k=1,\ldots,2^m\right\}$$ 
so we can apply the triangle inequality of the metric $\delta$ and we obtain that
$$\delta(x\oplus 0, x\oplus 1) \leq 2^m \frac{1}{4^m} = \frac{1}{2^m}$$
but this happening for every $m$ contradicts the fact that $\delta(x\oplus 0,x\oplus 1) >0$.
\end{proof}

\begin{remark}{\rm Note that the above proof does not work for $\LL_0$ because
$$\left\{ [t^\frown(0,0,0,\ldots), t^\frown(1,1,1,\ldots)]: t\in T,\ |t| = m\right\}$$
do not form consecutive intervals, that is, their left ends are not equal to any of their right ends, and so the triangle
inequality cannot be applied as in the proof above.}
\end{remark}

\section{A basic space of the form of the Ciesielski-Pol compact}\label{ZFCexample}

Remember that a set $S\subset \mathbb{R}$ is called a Bernstein set if both $S\cap P$ and $S\setminus P$ are nonempty  for every perfect set $P\subset \mathbb{R}$. The classical result of Bernstein is that such a set exist: it is constructed by transfinite induction by enumerating all possible perfect subsets of $\mathbb{R}$ as $\{P_\xi : \xi<\mathfrak{c}\}$ and at every step $\xi$ choosing new points $x_\xi,y_\xi\in P_\xi$ and declaring $x_\xi\in S$ and $y_\xi\not\in S$. A minor modification of this argument yields the existence of $\mathfrak c$ many disjoint Bernstein sets: write $\mathfrak c = \bigcup\{ I_\alpha : \alpha <\mathfrak c\}$ with $|I_\alpha|=\mathfrak c$, and assume that for every $\alpha$, $\{P_\xi : \xi\in I_\alpha\}$ enumerates all perfect subsets of $\mathbb{R}$; then at step $\xi\in I_\alpha$, choose new $x_\xi,y_\xi\in P_\xi$ and declare $x_\xi\in S_\alpha$ and $y_\xi\not\in S_\alpha$.\\

The basic space that we are going to construct is of the form $K = \bigcup_{n\in\mathbb{N}} A_n \cup B \cup \{\infty\}$ where the sets $A_n$ and the set $B$ are pairwise disjoint Bernstein subsets of $\mathbb{R}$. All points of $A=\bigcup_{n\in\N} A_n$ will be of course isolated, the space $A\cup B$ will be locally compact (its topology will be a refinement of the topology inherited from $\mathbb{R}$) and $K$ its one-point compactification. In order to describe completely our basic space we need to say which are the sets $C_x$ for $x\in B$ (that will provide a basis of neighborhoods of such $x\in B$: all $H\cup\{x\}$ where $H$ is cofinite in $C_x$)  and also which is the function $\psi:B\To \mathbb{N}$.\\

All topological notions below refer to the standard topology on $\R$.
Let $(F_\alpha)_{\alpha<\cc}$ be an enumeration of all 
sequences $(F_\alpha(n,m))_{n,m\in\N}$  of countable
subsets of $\R$ such that $F_\alpha(n,m)\subseteq A_n$  for each
$n,m\in\N$ and
$$\boxplus F_\alpha=\bigcap_{n\in\N}\bigcup_{m\in\N} {\overline {F_\alpha(n,m)}}$$
contains a perfect set.

We construct $\{x_\alpha: \alpha<\cc\}\subseteq B$, the sets $C_{x_\alpha}$ and 
and $\psi(x_\alpha)$ by induction on $\alpha<\cc$.
Given $\alpha$ we pick 
$$x_\alpha\in \boxplus F_\alpha\setminus \{x_\beta: \beta<\alpha\}.$$
We define $\psi(x_\alpha)=(m_n)_{n\in\N}$ to be such a sequence that
$x_\alpha\in \bigcap_{n\in\N}{\overline{F_\alpha(n,m_n)}}$ which exists since $x_\alpha\in \boxplus F_\alpha$.
Then take as $C_{x_\alpha}$ the terms of a  sequence which converges in $\R$ to $x_\alpha$ and such that for every $n$, $C_{x_\alpha}$ contains infinitely many elements from $F_\alpha(n,m_n)$. After the inductive procedure is finished, For the remaining elements
$x\in B\setminus \{x_\alpha : \alpha<\mathfrak{c}\}$, we define $\psi(x)$ to be any arbitrary value, and $C_x$ any sequence of elements of $A$ convergent to $x$.\\

We finally check that the key property (3) of basic spaces is satisfied. Suppose that we have $A_n = \bigcup X^n_m$ for every $n$. For every $n$, the set $\bigcup \overline{X^n_m}$ is a Borel set and it intersects every perfect set since it contains the Bernstein set $A_n$. Therefore $\bigcup \overline{X^n_m}$ is cocountable in $\mathbb{R}$ (every uncountable Borel set contains a perfect set). Therefore 
$$\bigcap_{n\in\N}\bigcup_{m\in\N} {\overline {X^n_m}}$$
is cocountable and in particular, contains a perfect set.
We choose countable sets $F(n,m)\subset X^n_m$ with $\overline{F(n,m)} = \overline{X^n_m}$. Then, since
$$\boxplus F_\alpha=\bigcap_{n\in\N}\bigcup_{m\in\N} {\overline {X^n_m}}$$
contains a perfect set, this sequence must appear in our enumeration as $F = F_\alpha$ for some $\alpha<\mathfrak{c}$. Let us see that $x=x_\alpha$ is the element of $B$ that we are looking for. Indeed, by the way we chose $C_{x_\alpha}$ and $\psi(x_\alpha)$, we know that $C_{x_\alpha}$ contains infinitely many elements from $F_\alpha(n,m_n) \subset X^n_{m_n} = X^n_{\psi(x_\alpha)[n]}$, for every $n$. This finishes the proof.

\section{A basic space from a ladder system under $\diamondsuit$}\label{diamondexample}

\begin{definition} $(D_\alpha)_{\alpha\in \omega_1}$ is called a $\diamondsuit$-sequence if and only if for every
$X\subseteq \omega_1$ the set
$$\{\alpha\in \omega_1: D_\alpha=X\cap\alpha\}$$
is stationary (i.e., intersects all closed in the order topology and unbounded subsets of $\omega_1$).
$\diamondsuit$ is a statement that a $\diamondsuit$-sequence exists.
\end{definition}

In the basic space that we construct now, $K= \bigcup_{n=1}^\infty A_n \cup B \cup C$, we will have
\begin{enumerate}
\item $A_n$ is the set of all countable ordinals of the form $\alpha + n$, where $\alpha$ is a limit ordinal. Hence $A = \bigcup_1^\infty A_n$ is the set of all countable successor ordinals.
\item $B$ is the set of all countable limit ordinals, except 0.
\item $C = \{\omega_1\}$.
\end{enumerate} 

The sets $\{C_x : x\in B\}$ will be a ladder system in $\omega_1$. That is, for every $x\in B$, $C_x =\{\beta_1,\beta_2,\ldots\}\subset A$ with $\beta_1<\beta_2<\cdots$ and $\sup\{\beta_n : n<\omega\} = x$. Once the ladder system is given, the topology considered on $K$ is such that each point of $A$ is isolated, a basis of neighborhoods of $x\in B$ are the sets $H\cup\{x\}$ with $H$ cofinite in $C_x$, and $K$ is the one-point compactification of the locally compact space $A\cup B$.\\

Now, we have to explain how to find a ladder system $\{C_x : x\in B\}$ and a function $\psi:B\To\mathbb{N}^\mathbb{N}$ so that the fundamental property (3) of a basic space is satisfied.

So let $(D_\alpha)_{\alpha<\omega_1}$ be a $\diamondsuit$-sequence. 
Let $x\in B$. Suppose first that:
\begin{enumerate}
\item $x\neq \omega$,
\item $D_x\subset A$,
\item $\sup(D_x\cap A_n) = x$ for every $n$, and 
\item there exists $f:\N\rightarrow\N$ such that $D_x\cap\omega=\{2^n3^{f(n)}: n=1,2,3,\ldots\}$.
\end{enumerate}
If all this conditions hold, we define $\psi(x)=f$ and  $C_x$ to be some increasing sequence of elements of $D_x\setminus \omega$ whose supremum is $x$ and which contains infinitely many elements of $A_n$ for every $n$. For the remaining $x\in B$ that do not satisfy the conditions above, we define $C_x$ and $\psi(x)$ in an
arbitrary way.\\

Now suppose that $A_n = \bigcup_m X_m^n$ as in condition (3) of a basic space. For every $n$, choose $m_n\in\N$ such that
$X^n_{m_n}$ is uncountable. Let $G_n$ be the set of all limit ordinals which are the supremum of some sequence contained in $X^n_{m_n}$. Notice that this is a closed and unbounded subset of $\omega_1$. Define
$$X=\{2^n3^{m_n}: 1\leq n <\omega\}\cup\bigcup_{n=1}^\infty (X_{m_n}^n\setminus\omega).$$
By the choice of $(D_\alpha)_{\alpha\in L(\omega_1)}$, there is $\alpha>\omega$, 
$\alpha\in \bigcap_{n=1}^\infty G_n$ such that $X\cap \alpha=D_\alpha$.
Then $x=\alpha$ is the element that we were looking for.

\end{document}